\documentclass[letterpaper,12pt,reqno]{amsart}
\usepackage{amssymb,latexsym, amsmath, amsxtra, amsthm,amsfonts}
\usepackage[noadjust]{cite}
\usepackage{url}
\usepackage[hidelinks]{hyperref}
\usepackage{graphicx}
\let\originalleft\left
\let\originalright\right
\renewcommand{\left}{\mathopen{}\mathclose\bgroup\originalleft}
\renewcommand{\right}{\aftergroup\egroup\originalright}

\newcommand{\es}[1]{\begin{equation}\begin{split}#1\end{split}\end{equation}}
\newcommand{\est}[1]{\begin{equation*}\begin{split}#1\end{split}\end{equation*}}
\newcommand{\pr}[1]{\left( #1\right)}

\newcommand{\pmd}[1]{\left| #1\right|}
\newcommand{\tn}[1]{\textnormal{#1}}
\newcommand{\Res}{\operatornamewithlimits{Res}}
\newcommand{\C}{\mathbb{C}}

\newtheorem{theorem}{Theorem}
\newtheorem{lemma}{Lemma}

\begin{document}

\title{An optimal choice of Dirichlet polynomials for the Nyman-Beurling criterion}
\author {S. Bettin, J. B. Conrey, D. W. Farmer}
\address{Sandro Bettin -- School of Mathematics, University of Bristol, Queens Ave, Bristol, BS8 1SN, UK}
\curraddr{Centre de Recherches Math\'ematiques - Universit\'e de MontrŽal, P.O. Box 6128, Centre-ville Station, Montr\'eal (Qu\'ebec), H3C 3J7, Canada}
\address{J. Brian Conrey -- American Institute of Mathematics, 360 Portage Avenue, Palo Alto, CA 94306 - 2244 USA}
\address{David W. Farmer -- American Institute of Mathematics, 360 Portage Avenue, Palo Alto, CA 94306 - 2244 USA}

\dedicatory{In memory of Professor A. A. Karatsuba on the 75th anniversary of his birth}


\maketitle
\begin{abstract}
We give a conditional result on the constant in the B\'aez-Duarte reformulation of the Nyman-Beurling criterion for the Riemann Hypothesis. We show that assuming the Riemann hypothesis and that $\sum_{\rho}\frac{1}{|\zeta'(\rho)|^2}\ll T^{3/2-\delta}$, for some $\delta>0$, the value of this constant coincides with the lower bound given by Burnol.
\end{abstract}
\section{Introduction}
The Nyman-Beurling-B\'aez-Duarte approach to the Riemann hypothesis asserts that the Riemann hypothesis is true if and only if 
\est{
\lim_{N\to \infty} d_N^2=0,
}
where 
\est{
d_N^2=\inf_{A_N} \frac{1}{2\pi} \int_{-\infty}^\infty |1-\zeta A_N(1/2+it)|^2 \frac{dt}{\frac 14 +t^2}
}
and the infimum is over all Dirichlet polynomials $A_N(s)=\sum_{n=1}^N\frac{a_n}{n^s}$ of length $N$ (see~\cite{Bag} for a nice account of this). 

An open question is to determine what the rate of convergence of $d_n$ to zero is, assuming the Riemann hypothesis. Balazard and de Roton showed that, if the Riemann hypothesis is true, then
\est{
d_{N}^2\ll \frac{(\log\log N)^{\frac52+\varepsilon}}{\sqrt{\log N}},
}
for all $\varepsilon>0$. On the other hand B\'aez-Duarte, Balazard, Landreau and Saias~\cite{BBLS00,BBLS05} showed (unconditionally) that $d_N^2$ can not decay faster than a constant times $\frac1{\log N}$. More precisely, they showed that
\est{
 \liminf_{N\to \infty } d_N^2 \log N \ge \sum_{\Re (\rho)=1/2} \frac{1}{|\rho|^2},
}
where here and in the following the sum is restricted to distinct zeros of the Riemann zeta function on the critical line. The constant was later improved by Burnol~\cite{Bur} who showed
\est{
\liminf_{N\to \infty } d_N^2 \log N \ge \sum_{\Re (\rho)=1/2} \frac{m(\rho)^2}{|\rho|^2},
}
where $m(\rho)$ denotes the multiplicity of $\rho$. This lower bound is believed to be optimal and one expects that 
\es{\label{asm}
d_N^2 \sim \frac1{\log N}\sum_{\Re (\rho)=1/2} \frac{m(\rho)^2}{|\rho|^2}.
}
Notice that under the Riemann hypothesis, one has
\est{
\sum_{\Re (\rho)=1/2} \frac{m(\rho)}{|\rho|^2}=2+\gamma-\log 4\pi
}
and in particular, if all the non-trivial zeros of $\zeta(s)$ are simple, then~\eqref{asm} can be rewritten as
\est{
d_N^2 \sim \frac{2+\gamma-\log 4\pi}{\log N}.
}

It is the purpose of this note to prove~\eqref{asm} under the Riemann Hypothesis and assuming a mild condition on the growth of the mean value of  $\frac1{|\zeta'(\rho)|^2}$ over the non-trivial zeros $|\rho|\leq T$ of $\zeta(s)$. This will be achieved by using the Dirichlet polynomial 
\est{
V_N(s):=\sum_{n=1}^N\pr{1-\frac{\log n}{\log N}}\frac{\mu(n)}{n^s }.
}
\begin{theorem}\label{tr}
If the Riemann hypothesis is true and if 
\es{\label{cond}
\sum_{|\Im(\rho)|\leq T}\frac1{|\zeta'(\rho)|^2}\ll T^{\frac32-\delta}
}
for some $\delta>0$, then 
\est{
\frac{1}{2\pi} \int_{-\infty}^\infty |1-\zeta V_N(1/2+it)|^2 \frac{dt}{\frac 14 +t^2}\sim \frac{2+\gamma-\log 4\pi}{\log N}. 
}
\end{theorem}
The condition~\eqref{cond} implicitly assumes that the zeros of the Riemann zeta function are all simple. Moreover, this upper bound is ``mild'' in the sense that a conjecture, due to Gonek and recovered by a different heuristic method of Hughes, Keating, and O'Connell~\cite{HKO}, predicts that 
\est{
\sum_{|\rho|\leq T}\frac1{|\zeta'(\rho)|^2}\sim \frac 6{\pi^3} T.
}

We remark that Theorem~\ref{tr} is in contrast to what one might have expected after viewing the graphs of Landreau and Richards~\cite{LR} which at first sight suggest that $V_N$ is not optimal.

This behaviour of the Riemann zeta function resembles that of polynomials. In fact, Grenander and Rosenblatt~\cite{GR} (see also Theorem 2.1 in~\cite{Bur}) showed that  for a polynomial $P(z)$ one has that the zeros of $P$ are all located outside or on the unit circle if and only if $\lim_{N\rightarrow\infty}{\delta_N}=0$, where
\est{
\delta_N^2=\frac1{2\pi}\inf_{Q_N}\int_{0}^{2\pi}\pmd{1-P\pr{z}\,Q_N\pr{z}}^2\,\tn{d}\theta,
}
where $z=e^{i\theta}$ and the infimum is over polynomials $Q_N$ of degree at most $N$. Moreover, if this happens, then 
\est{
\lim_{N\rightarrow\infty} N\delta_N^2 =\sum_{|\rho|=1}{m(\rho)^2},
}
where the sum is restricted to the distinct zeros $\rho$ of $P(z)$ lying on the unit circle and $m(\rho)$ is again the multiplicity of $\rho$. 

This analogy seems to apply also to the choices of optimal polynomials. 
\begin{theorem}\label{tp}
Let $P(z)$ be a polynomial whose zeros are all simple and lie outside or on the unit circle. Let 
\es{\label{defw}
W_N(z):= \sum_{n=0}^N \pr{1-\frac nN}{a_n}z^n,
}
where 
\est{
\frac1{P(z)}=\sum_{n\geq0}a_nz^n
}
is the Taylor expansion in $x=0$ of the inverse of $P(z)$ (i.e. it is the formal power series inverse of $P(z)$). Then
\est{
\frac1{2\pi}\int_{0}^{2\pi}\pmd{1-P\pr{z}W_N\pr{z}}^2\,\tn{d}\theta\sim\frac1{N}\sum_{|\rho|=1}{m(\rho)^2},
}
where $z=e^{i\theta}$.
\end{theorem}

We remark that the proofs of Theorem~\ref{tr} and~\ref{tp} are very similar, the main difference being that the Riemann zeta function has infinitely many zeros. This generates some issues concerning the convergence of certain sums of $\frac1{\zeta'(\rho)}$, which force us to assume condition~\eqref{cond}.

\section{Polynomials}
\begin{lemma}\label{lp}
Let $P(s)$ be a polynomial with $P(0)\neq0$. We have 
\est{
W_N(s)=\frac 1{P(s)}\pr{1+\frac{s}{N}\frac{P'}{P}(s)}-\frac {s}{N}Y_N(s),
}
where $W_N(s)$ is defined in~\eqref{defw},
\est{
Y_N(s):=\sum_{\rho}\Res_{z=\rho}\frac{{s}^{N}}{P(z)(z-{s})^2z^{N}}
}
and the sum is over distinct zeros $\rho$ of $P(z)$.
\end{lemma}
\begin{proof}
Since $P(0)\neq0$, we can take an $\varepsilon>0$ such that all the zeros of $P(z)$ lie outside of the circle $|z|=\varepsilon$. Now, observe that we can assume $0<|s|<\varepsilon$, since the result will then extend to all $\C$ by analytic continuation. Denoting by $\mathcal{C}_y$ the circle of radius $y>0$ (oriented in the positive direction), by the residue theorem we have that 
\est{
a_n=\frac1{2\pi i}\int_{\mathcal{C}_\varepsilon} \frac1{P(z)}\frac{\tn{d}z}{z^{n+1}},
}
therefore 
\est{
W_N(s)=\frac1{2\pi i}\int_{\mathcal{C}_\varepsilon}\frac1{P(z)} \sum_{n=0}^N \pr{1-\frac nN}\pr{\frac{s}{z}}^n \frac{\tn{d}z}{z}.\\
}
Now,
\est{
\sum_{n=0}^N\pr{1-\frac n N}z^n=-\frac1N\frac{z-z^{N+1}}{(1-z)^2}+\frac1{1-z}
}
and thus
\est{
W_N(s)=\frac1{2\pi i}\int_{\mathcal{C}_\varepsilon}\frac1{P(z)} \pr{-\frac1N\frac{sz^N-{s}^{N+1}}{(z-{s})^2z^{N}}+\frac1{z-{s}}}{\tn{d}z}.
}
Now, by the residue theorem
\est{
\frac1{2\pi i}\int_{\mathcal{C}_\varepsilon}\frac1{P(z)} \pr{-\frac1N\frac{s}{(z-{s})^2}+\frac1{z-{s}}}{\tn{d}z}=\frac1{P(s)}\pr{1+\frac{s}{N}\frac{P'(s)}{P(s)}},
}
whereas, moving the line of integration to $\mathcal{C}_y$ and letting $y$ tend to infinity, one has that
\est{
\frac1{2\pi i N}\int_{\mathcal{C}_\varepsilon}\frac1{P(z)}\frac{{s}^{N+1}}{(z-{s})^2z^{N}}{\tn{d}z}=-\frac {s}{N}Y_N(s)
}
and the Lemma follows.
\end{proof}

\begin{proof}[Proof of Theorem~\ref{tp}]
Let $\delta>1$ be such that $P(s)$ does not have any zero on $1<|s|\leq\delta$. We have
\est{
\frac1{2\pi}\int_{0}^{2\pi}\pmd{1-P\pr{z}W_N\pr{z}}^2\,\tn{d}\theta&=\frac1{2\pi i}\int_{\mathcal{C}_1}\pr{1-P\pr{s}W_N\pr{s}}\pr{1-\overline{P}\pr{\frac1s}\overline{W}_N\pr{\frac1s}}\,\frac{\tn{d}s}s\\
&=\frac1{2\pi i}\int_{\mathcal{C}_{\delta}}\pr{1-P\pr{s}W_N\pr{s}}\pr{1-\overline{P}\pr{\frac1s}\overline{W}_N\pr{\frac1s}}\,\frac{\tn{d}s}s.\\
}
Therefore, by Lemma~\ref{lp}, this is
\est{
\frac1{2\pi i N^2}\int_{\mathcal{C}_\delta}\pr{\frac{P'}{P}(s)-P(s)Y_N(s)} \pr{\frac{\overline{P}'}{\overline{P}}\pr{\frac1s}-\overline {P}\pr{\frac1s}\overline{Y}_N\pr{\frac1s}}\,\frac{\tn{d}s}{s}.\\
}
Now, for $|s|=\delta$ one has
\est{
Y_N(s)\overline{Y}_{N}\pr{\frac1s}&=O(1),
}
therefore 
\est{
\frac1{2\pi i N^2}\int_{\mathcal{C}_\delta}\pr{\frac{P'}{P}(s)\frac{\overline{P}'}{\overline{P}}\pr{\frac1s}+P(s)Y_N(s)\overline {P}\pr{\frac1s}\overline{Y}_N\pr{\frac1s}}\,\frac{\tn{d}s}{s}=O\pr{\frac1{N^2}}.\\
}
Moreover for $s\in\mathcal{C}_\delta$ one has that $\overline{Y}_N\pr{\frac1s}=O\pr{\delta^{-N}}$, thus
\est{
-\frac1{2\pi i N^2}&\int_{\mathcal{C}_\delta}\pr{\frac{P'}{P}(s)\overline {P}\pr{\frac1s}\overline{Y}_N\pr{\frac1s}}\,\frac{\tn{d}s}{s}=O\pr{\delta^{-N}/N^2}.\\
}
Finally, by the residue theorem,
\est{
-\frac1{2\pi i N^2}&\int_{\mathcal{C}_\delta}P(s)Y_N(s)\frac{\overline{P}'}{\overline{P}}\pr{\frac1s}\,\frac{\tn{d}s}{s}=\\
&=-\frac1{N^2}\sum_{|\rho|=1}\Res_{s=\rho}P(s)Y_N(s)\frac{\overline{P}'}{\overline{P}}\pr{\frac1s}\frac1s+\\
&\quad-\frac1{2\pi i N^2}\int_{\mathcal{C}_{\frac1\delta}}P(s)Y_N(s)\frac{\overline{P}'}{\overline{P}}\pr{\frac1s}\,\frac{\tn{d}s}{s}\\
&=-\frac1{N^2}\sum_{|\rho|=1}\Res_{s=\rho}P(s)Y_N(s)\frac{\overline{P}'}{\overline{P}}\pr{\frac1s}+O\pr{\delta^{-N}/N^2}.\\
}
The theorem then follows by observing that
\est{
\Res_{s=\rho}P(s)Y_N(s)\frac{\overline{P}'}{\overline{P}}\pr{\frac1s}\frac1s=-N+O(1).
}
\end{proof}

\section{The Riemann zeta-function}
We start with the following lemma, which is the analogue of Lemma~\ref{lp}. We remark that this lemma is unconditional.

\begin{lemma}\label{lr1}
If $0< \Re \pr{s} < 1$, then
\est{
V_N(s)=\frac{1}{\zeta(s)} \left(1-\frac 1{\log N} \frac{\zeta'}{\zeta}(s)\right) +\frac{1}{\log N} \sum_{\rho}
R_N(\rho,s) +\frac{1}{\log N} F_s(1/N),
}
where the sum is over distinct non-trivial zeros $\rho$ of $\zeta(s)$ with
\est{
R_N(\rho,s)= \Res_{z=\rho} \frac{N^{z-s}}{\zeta(z)(z-s)^2},
}
and where
\est{
F_s(z)=\pi z^s \sum_{n=1}^\infty \frac{(-1)^n (2\pi)^{2n+1} z^{2n}}{(2n)! \zeta(2n+1) (2n+s)^2}
}
is an entire function of $z$.
\end{lemma}
\begin{proof}
We have
\est{
V_N(s)=\frac 1 {\log N}\frac{1}{2\pi i}\int_{(2)}\frac{N^w}{\zeta(s+w)}\frac{dw}{w^2},
}
where we use the notation $\int_{(c)}$ to mean an
integration up the vertical line from $c-i\infty$ to $c+i\infty$. Now we move the path of integration to $\Re (w) = -\Re (s) -2M -1$
for a large integer $M$. The residue at $w=\rho-s$ is $R_N(\rho,s)/\log N$. The residue at $s+w=-2 n$ is
\est{
\frac{N^{-2 n-s}}{\zeta'(-2 n)(2 n+s)^2 \log N}
}
and the integral on the new path is $\ll N^{-2M-1}$. Letting $M \to \infty$ and using
\est{
\zeta'(-2n)=\frac{(-1)^n \pi (2n)! \zeta(2 n+1)}{(2\pi)^{2 n+1}}
}
we obtain the result.
\end{proof}

\begin{lemma}\label{lr2}
Let $\varepsilon>0$. Assume the Riemann hypothesis and that all the zeros of $\zeta(s)$ are simple. Then, if condition~\eqref{cond} holds, for $\Re(s)=\frac12\pm\varepsilon$ one has
\es{\label{elr2}
\sum_{\rho}R_N(\rho,s)\ll N^{\mp\varepsilon}\pmd{s}^{\frac34-\frac\delta2+\varepsilon}.
}
\end{lemma}
\begin{proof}
Firstly observe that, by the Cauchy-Schwartz inequality,~\eqref{cond} implies
\est{
\sum_{|\rho|\leq T}\frac{1}{\pmd{\zeta'(\rho)}}\ll \sqrt{ N(T)\sum_{|\rho|\leq T}\frac{1}{\pmd{\zeta'(\rho)}^2}}\ll T^{\frac54-\frac{\delta }2}\sqrt{\log T},
}
since 
\est{
N(T):=\frac12\sum_{|\rho|\leq T}1=\frac{T}{2\pi}\log\frac T{2\pi e}+O\pr{\log T}. 
}
Therefore, by partial summation, we have that the series
\est{
\sum_{\rho}\frac{1}{|\zeta'(\rho)||\rho|^{\alpha}}
}
is convergent for any $\alpha>\frac54-\frac{\delta}2$. Now, for a simple zero $\rho$, we have
\est{
R_N(\rho,s)=\sum_{\rho}\frac{N^{\rho-s}}{\zeta'(\rho)(\rho-s)^2}.
}
Therefore
\es{\label{e2lr2}
N^{\pm\varepsilon} \sum_{\rho}R_N(\rho,s)&\ll  \sum_{|\rho-s|<\frac{|\rho|}2}\frac{1}{|\zeta'(\rho)||\rho-s|^2}+ \sum_{|\rho-s|\geq\frac{|\rho|}2}\frac{1}{|\zeta'(\rho)||\rho-s|^2}\\
&\ll  \sum_{|\rho-s|<\frac{|\rho|}2}\frac{1}{|\zeta'(\rho)||\rho-s|^2}+ \sum_{|\rho-s|\geq\frac{|\rho|}2}\frac{1}{|\zeta'(\rho)||\rho|^2}\\
&\ll  \sum_{|\rho-s|<\frac{|\rho|}2}\frac{1}{|\zeta'(\rho)||\rho-s|^2}+1.\\
}
Now, by the Cauchy-Schwartz inequality,
\est{
\sum_{|\rho-s|<\frac{|\rho|}2}\frac{1}{|\zeta'(\rho)||\rho-s|^2}\ll\sqrt{\pr{\sum_{|\rho|<{2}|s|}\frac{1}{|\zeta'(\rho)|^2}}\
\pr{\sum_{|\rho|<2|s|}\frac{1}{|\rho-s|^4}}}\ll |s|^{\frac 34-\frac\delta2+\varepsilon},
}
since, by partial summation,
\est{
\sum_{|\rho|<2|s|}\frac{1}{|\rho-s|^4}\ll\log(|s|+2).
}
This completes the proof of the lemma.
\end{proof}

\begin{proof}[Proof of Theorem~\ref{tr}]
We have
\est{
\frac{1}{2\pi}\int_{-\infty}^\infty &|1-\zeta V_N(1/2+it)|^2 \frac{dt}{1/4+t^2}\\
&=\frac1{2\pi i}\int_{\pr{\frac12}}\pr{1-\zeta V_N(s)}\pr{1-\zeta V_N(1-s)}\,\frac{ds}{s(1-s)}\\
&=\frac1{2\pi i}\int_{\pr{\frac12-\varepsilon}}\pr{1-\zeta V_N(s)}\pr{1-\zeta V_N(1-s)}\,\frac{ds}{s(1-s)}.\\
}
By Lemma~\ref{lr1}, this is
\es{\label{tt}
\frac1{\log^2 N}\frac1{2\pi i}\int_{\pr{\frac12-\varepsilon}}&\pr{\frac{\zeta'}{\zeta^2}(s) -\sum_{\rho}
R_N(\rho,s) -F_s\pr{\frac1N}}\times\\
&\times \pr{\frac{\zeta'}{\zeta^2}(1-s) - \sum_{\rho}
R_N(\rho,1-s) -F_{1-s}\pr{\frac1N}} \frac{\zeta(s)\zeta(1-s)}{s(1-s)}\,{ds}.\\
}
Now, we have
\est{
\frac1{\log^2 N}\frac1{2\pi i}&\int_{\pr{\frac12-\varepsilon}} \sum_{\rho_1,\rho_2}
R_N(\rho_1,s)R_N(\rho_2,1-s) \frac{\zeta(s)\zeta(1-s)}{s(1-s)}\,{ds}\\
&\ll\frac1{\log^2 N}\int_{\pr{\frac12-\varepsilon}} {\sum_{|\rho-s|<\frac{|\rho|}2}
\frac{1}{|\zeta'(\rho)||\rho-s|^2} \,\frac{|ds|}{|s|^{\frac54+\frac{\delta}2-5\varepsilon}}}+O\pr{\frac1{\log^2 N}},\\
}
where we used~\eqref{elr2},~\eqref{e2lr2} and the bound $\zeta\pr{\frac{1}2\pm \varepsilon\pm it}\ll |t|^{2\varepsilon}$ (which is a consequence of the Lindel\"of hypothesis). Reversing the order of summation and integration, we have that this is bounded by
\est{
\frac1{\log^2 N}\sum_{\rho}\frac{1}{|\zeta'(\rho)|}&\int_{\pr{\frac12-\varepsilon}+i\pr{\Im(\rho)-\frac{\pmd{\rho}}2}}^{\pr{\frac12-\varepsilon}+i\pr{\Im(\rho)+\frac{\pmd{\rho}}2}}
 \,\frac{|ds|}{|\rho-s|^2|s|^{\frac54+\frac{\delta}2-5\varepsilon}}+O\pr{\frac1{\log^2 N}}\\
 &\ll\frac1{\log^2 N}\sum_{\rho}\frac{1}{|\zeta'(\rho)||\rho|^{\frac54+\frac{\delta}2-5\varepsilon}}\ll\frac1{\log^2 N},
}
if $\varepsilon<\frac{\delta}{10}$.

Now, by Lemma~\ref{lr2} and the trivial estimate $F_s(z)=O\pr{N^{-\frac 52}}$, all the other terms in~\eqref{tt} are trivially $O\pr{\frac1{\log^2 N}}$ apart from 
\es{\label{etr}
-\frac1{\log^2 N}\frac1{2\pi i}\int_{\pr{\frac12-\varepsilon}} \frac{\zeta'}{\zeta}(1-s) \sum_{\rho}
R_N(\rho,s) \frac{\zeta(s)}{s(1-s)}\,{ds}.\\
}
The integrand has a double pole at every zero $\rho$ of residue 
\est{
\operatorname{Res}_{s=\rho}\pr{{\frac{\zeta'}{\zeta}(1-s) \sum_{\rho}
R_N(\rho,s) \frac{\zeta(s)}{s(1-s)}}}&=
\frac{\log N-\frac12\frac{\zeta''\pr{\rho}}{\zeta'(\rho)}+\frac{\chi'}{\chi}(\rho)+\frac{1-2\rho}{|\rho|^2}}{|\rho|^2}\\
&=\frac{\log N}{|\rho|^2}+O\pr{\frac{1}{|\rho|^{2-\varepsilon}|\zeta'(\rho)|}+\frac1{|\rho|^2}},
}
where we used the bound $\zeta''\pr{\frac12+it}\ll |t|^{\varepsilon}$, which follows from the Lindel\"of hypothesis and Cauchy's estimate for the derivatives of a holomorphic function. It follows that moving the line of integration in~\eqref{etr} to $\Re(s)=\frac12+\varepsilon$ we get that the integral is equal to
\est{
\frac1{\log N}\sum_{\rho}\frac{1}{|\rho|^2}+O\pr{\frac1{\log^2N}},
}
and Theorem~\ref{tr} then follows.
\end{proof}

 \appendix

\ \\
 
\end{document}